\documentclass[10pt,a4paper]{article}
\usepackage[utf8]{inputenc}
\usepackage{amsmath}
\usepackage{amsfonts}
\usepackage{amssymb}
\usepackage{amsthm}
\newtheorem{theorem}{Theorem}

\newtheorem{lemma}{Lemma}
\newtheorem{definition}{Definition}

\begin{document}

\title{Nonfreeness of algebras of symmetric Hilbert modular forms of even weight for $\mathbb{Q}(\sqrt{d})$ where $d>5$ }
\author{Ekaterina Stuken }

\maketitle
\begin{abstract}
We study the algebras of symmetric Hilbert modular forms of even weight for $\mathbb{Q}(\sqrt{d})$, considering them as modular forms for the orthogonal group of the lattice with signature (2,2). Comparing the volume of the corresponding symmetric domain with the volume of the Jacobian of the generators of these algebras, we prove that for all d, except for d=2, 3, 5 these algebras can't be free.
\end{abstract}
\section*{Introduction}
Consider the quadratic field $\mathcal{K}=\mathbb{Q}(\sqrt{d})$, it's discriminant $D=d \Leftrightarrow \ d=1 \pmod{4}$ otherwise $D=4d$, $d=p_1\cdot\ldots p_k$, $p_i$ are different primes, $p_1<p_2<\ldots <p_k$. It is well known that the extended Hilbert modular group $\widehat{Hilb}(\mathcal{K})=\widehat{GL_2(\mathcal{O}_\mathcal{K})\ltimes \sigma}$, where $\langle\sigma\rangle=Gal(\mathcal{K}/\mathbb{Q})$ is isomorphic the automorphism group of even lattice $L$ of signature $(2,2)$ and discriminant $D$ \cite{123}. If $d=1\pmod{4}$, $L=U\oplus \begin{pmatrix}
2 & 1  \\
1 & \frac{1-d}{2} \\
\end{pmatrix}$, if $d=3\pmod{4}$, $L= U\oplus \begin{pmatrix}
-2 & 0  \\
0 & 2d \\
\end{pmatrix}$. Besides, the subgroup of $\widehat{Hilb}(\mathcal{K})$, preserving the product of upper half-planes, coincides with the subgroup $O^+(L)$ of index 2 in $O(L)$. $O^+(L)=\lbrace g\in O(L) | \det g \cdot n(g)=1 \rbrace$ where $n(g)$ -- the spinor norm of $g$. 

The lattice $L$ defines the Cartan domain of type IV:
$$\Omega_L=\lbrace[\omega]\in \mathbb{P}(L\otimes \mathbb{C}); (\omega,\omega)_L=0,\ (\omega,\overline{\omega})_L>0\rbrace$$
It has two connected components $D_L$ and $D'_L$. $O^+(L)$ corresponds to the subgroup, preserving $D_L$. 
The affine cone over $D_L$ in $L\otimes \mathbb{C}$ will be denoted by $D_L^\bullet$.
\begin{definition}
A modular form of weight $k$ with respect to $O(L)$ and with character $\chi:O(L)\rightarrow \mathbb{C}^*$ is a holomorphic map $f:D_L^\bullet\rightarrow\mathbb{C}$ such that:\\
1) $f(tz)=t^{-k}f(z)$, $t\in\mathbb{C}^*$\\
2) $f(g(z))=\chi(g)f(z)$, $g\in O(L)$
\end{definition}
We notice that there are no modular forms of odd weight with trivial character. It follows from the second condition from the definition of modular forms: $-id\in O^+(L)$, because $\det (-id)=1,\ n(-id)=1$, since $-id=r_{e_3}r_{e_4}r_{e_1+e_2}r_{e_1-e_2}$ when $d\neq 1 \pmod{4}$, $-id=r_{e_3-2e_4}r_{e_3}r_{e_1+e_2}r_{e_1-e_2}$ when $d=1\pmod{4}$ ($e_i$ -- basis of $L$, $r_{e_i}$ -- the reflection in $e_i$). Besides, $\chi(-id)=1$, so $f(-z)=f(z)$. But from the first condition we get that $f(-z)=-f(z)$, so there is a contradiction. So we will consider only forms of even weight and trivial character.

The proof of the main result will be based on the Bruinier's theorem \cite{bru}. 
\begin{theorem}
Let $L$ be an even lattice of signature $(2,2)$ such that the dimension of maximal isotropic subspace $L\otimes \mathbb{Q}$ is equal to 1. If $F$ is a meromorphic $O^+(L)$-modular form of weight $K$, vanishing on all mirrors of reflections with multiplicity 1, then 
\begin{equation}
Vol(O^+(L))\cdot K =\frac{1}{2} \sum Vol(O^+(e^\perp\cap L))
\end{equation}
where $e$ -- roots of $L$, the sum is taken over all conjugacy classes.
\end{theorem}

\begin{theorem}[\cite{aok}] Assume the algebra of $O^+(L)$-modular forms is free with generators of weights $k_1,\ldots,k_t$. Then there exists $F$ -- $O^+(L)$-modular form of weight $2+k_1+\ldots+k_t$ vanishing with multiplicity one on all mirrors of reflections of $L$ and only on them.
\end{theorem}
It is shown in the proof that the modular form $F$ that is equal to the Jacobian of the generators satisfies these conditions.
\begin{theorem}
Assume $d=1 \pmod 4$. Then the algebra of $O^+(L)$-modular forms of even weight can be free only if $$ d=5,\ d=13,\ d=3 \cdot 7$$
\end{theorem}

\begin{theorem}
Assume  $d=3 \pmod 4$. Then the algebra of $O^+(L)$-modular forms of even weight can be free only if
$d=3$.
\end{theorem}

\begin{theorem}
Assume $d=2 \pmod 4$. Then the algebra of $O^+(L)$-modular forms of even weight can be free only if $d=6$. 
\end{theorem}

The idea of the proof is the following. We find the lower bound for the left hand side of formula (1) and an upper bound for the right hand side. Since the l.h.s. grows faster, we obtain that the formula (1) doesn't hold when the discriminant is big enough. We notice, that there are exactly three generators for the algebra of $O(L)$-modular forms, and since the weight of each is at least two, $K\geq 8$.
Besides, because $[O(L):O^+(L)]=2$, $Vol(O^+(L))=2Vol(O(L))$, and since both sides of (1) are multiplied by 2, we can consider $Vol(O(L))$ instead of $Vol(O^+(L))$. 

 
\section*{The number of classes in the genus}

\begin{theorem}\cite{kne}
Let $L$ be an indefinite lattice. If the genus of $L$ contains more than one class, then there exists a prime $p$ such that $L\otimes \mathbb{Z}_p$ can be diagonalised and the diagonal entries consist of distinct powers of $p$.
\end{theorem}

\begin{lemma}
Each even integral lattice of signature $(2,\ 1)$ with fixed square-free discriminant $d=-2p_1\ldots p_m=:2d'$ and with fixed Hasse invariant belongs to the same class. 
\end{lemma}
\begin{proof}
The discriminant of $L$ is square-free so it follows from Kneser's theorem that the class of $L$ coincides with the genus of $L$. So it suffices to prove that all such lattices are locally equivalent. 
For each $p\neq p_i$ the lattices are equivalent since they are unimodular. For $p=p_i$ there exist exactly 4 different lattices over $\mathbb{Z}_{p_{i_j}}$ ($\langle 1 \rangle \oplus \langle 1 \rangle \oplus \langle p \rangle,\  \langle 1 \rangle \oplus \langle r \rangle \oplus \langle p \rangle,\ \langle 1 \rangle \oplus \langle 1 \rangle \oplus \langle rp \rangle,\ \langle 1 \rangle \oplus \langle r \rangle \oplus \langle rp \rangle  $, where $r$ is a fixed quadratic non-residue mod $p$. The Hasse invariant of the first and the third lattices equals to 1, of the second and the fourth to -1. Since this invariant is fixed, we have one pair of lattices left. The lattices in each pair have different discriminants, and because it is also fixed we obtain one lattice left.
We need to check the equivalence over $\mathbb{Z}_2$. The algorithm for doing it can be found in \cite{con} and \cite{con1}. There are 2 possibilities: $U\oplus \langle 2d'\rangle$ and $\begin{pmatrix}
2 & 1  \\
1 & 2 \\
\end{pmatrix} \oplus \langle -6d'\rangle$. The Conway-Sloane notation for them is $1^{+2}_{II}2^{\omega(d')}_{d'}$ and $1^{-2}_{II}2^{\omega(-3d')}_{-3d'}=1^{-2}_{II}2^{-\omega(d')}_{-3d'}$ respectively. Consider the first lattice. It has adjacent scales (1 and 2) and different types (the unimodular part is even, 2-modular part is odd). So we can apply sign walking, i.e. change upper signs and add 4 to the lower oddity. Since $d'+4=-3d'$ for $d'=\pm1,\pm3 \pmod{8}$, we get that $1^{+2}_{II}2^{\omega(d')}_{d'}\sim 1^{-2}_{II}2^{-\omega(d')}_{-3d'}$. So, these lattices are equivalent over $\mathbb{Z}_p$ for each $p$, i.e. belong to the same genus.
\end{proof}
\begin{lemma}
Each even integral lattice of signature $(2,\ 1)$ with fixed square-free discriminant $d=-4p_1\ldots p_m=:4d'$ and with fixed Hasse invariant belongs to the same class. 
\end{lemma}
\begin{proof}
The invariant factors are $1,\ 1,\ 4d'$, otherwise the lattice couldn't be even, since it would have an unimodular component of odd dimension. It follows from Kneser's theorem that the class of $L$ coincides with the genus of $L$. So it suffices to prove that all such lattices are locally equivalent. 
For each $p\neq 2$ the proof is similar to the proof of lemma 1. 
Over $\mathbb{Z}_2$ there are two possibilities: $U\oplus \langle 4d'\rangle$ and $\begin{pmatrix}
2 & 1  \\
1 & 2 \\
\end{pmatrix} \oplus \langle -12d'\rangle$. They differ by Hasse invariant, because $\epsilon_2(U\oplus \langle 4d' \rangle)=1 \Leftrightarrow d'=1 \pmod{4} \Leftrightarrow \epsilon_2\left(\begin{pmatrix}
2 & 1  \\
1 & 2 \\
\end{pmatrix} \oplus \langle -12d' \rangle \right)=-1$. Since it is fixed, the lemma is proved.
\end{proof}
\begin{lemma}
Each even integral lattice of signature $(2,\ 1)$ with fixed square-free discriminant $d=-8p_1\ldots p_m=:-8d'$, fixed Hasse invariant and invariant factors $2,\ 2,\ 2d'$ belongs to the same class. 
\end{lemma}
\begin{proof}
 It follows from Kneser's theorem that the class of $L$ coincides with the genus of $L$. So it suffices to prove that all such lattices are locally equivalent. For each $p\neq 2$ the proof is similar to the proof of lemma 1. Over $\mathbb{Z}_2$ there are two possibilities: $U(2)\oplus \langle 2d' \rangle$ and $\begin{pmatrix}
4 & 2  \\
2 & 4 \\
\end{pmatrix} \oplus \langle -6d' \rangle$. They differ by Hasse invariant, because $\epsilon_2(U(2)\oplus \langle 2d' \rangle) = 1 \Leftrightarrow d'=1 \pmod{4} \Leftrightarrow$ $ \epsilon_2 \left(\begin{pmatrix}
4 & 2  \\
2 & 4 \\
\end{pmatrix} \oplus \langle -6d' \rangle \right) =-1$. Since it is fixed, the lemma is proved.
\end{proof}

\begin{lemma}
There exist two classes of ternary even lattices $M$ of signature $(2,\ 1)$ and discriminant $d=-16p_1\ldots p_m=:-16d'$ with invariant factors $2,\ 2,\ 4d'$ and fixed Hasse invariant.
\end{lemma}
\begin{proof}
It follows from Kneser's theorem that the class of $L$ coincides with the genus of $L$. So it suffices to prove that all such lattices are locally equivalent. 
For each $p\neq 2$ the proof is similar to the proof of lemma 1. 
There are 2 possibilities for $p=2$: the Gram matrix of $M$ can be the doubled matrix of even or odd lattice. Let $M$ be the doubled matrix of an even lattice, then it could be: $U(2)\oplus \langle 4d' \rangle$ or $\begin{pmatrix}
4 & 2  \\
2 & 4 \\
\end{pmatrix} \oplus \langle -12d' \rangle$. In the Conway notation they have the form $2_{II}^{+2}4_{d'}^{\omega(d')}$ and $2_{II}^{-2}4_{-3d'}^{-\omega(d')}$. We can apply "sign walking" and because $d'+4=-3d' \pmod{8}$, we obtain that this lattices are equivalent over $\mathbb{Z}_2$. Since they are equivalent over all $\mathbb{Z}_p$, they belong to the same class.\\
Let $M$ be the doubled matrix of an odd lattice, then $M=\langle 2\epsilon_1\rangle \oplus \langle 2\epsilon_2\rangle \oplus \langle -4d'\epsilon_1\epsilon_2\rangle$. 
We note that the Hasse invariant is uniquely defined by $\epsilon_1,\ \epsilon_2,\ d'$. Going through all the possibilities on $\epsilon_1,\ \epsilon_2,\ d' \ \pmod{8}$, we get that for each $\epsilon_1$, $\epsilon_2$ the fixed $d'$ and $\epsilon$ uniquely determine the lattice:
\begin{center}
\begin{tabular}{|c|c|c|}
\hline
$d' \pmod{8}$ & $\epsilon \pmod{8}$ & $M$ \\
\hline
1 & 1 & $[2^{+2}4^{+1}]_1$ \\
1 & -1 & $[2^{+2}4^{+1}]_{-3}$ \\
-1 & 1 & $[2^{+2}4^{+1}]_3$ \\
-1 & -1 & $[2^{+2}4^{+1}]_{-1}$ \\
3 & 1 & $[2^{+2}4^{-1}]_{-1}$ \\
-3 & 1 & $[2^{+2}4^{-1}]_{-3}$ \\
3 & -1 & $[2^{+2}4^{-1}]_3$ \\
-3 & -1 & $[2^{+2}4^{-1}]_{1}$ \\
\hline
\end{tabular}
\end{center}
Thus, there are two conjugacy classes: the first corresponds to the doubled even lattice, the second to the doubled odd lattice.
\end{proof}

\section*{Reflections}

Let $L$ be an even lattice, $e\in L$ -- primitive $2k$-root of $L$. $\langle e \rangle = \mathbb{Z}e$, $\langle e^\perp \rangle = L\cap e^\perp$.

\begin{lemma}
There are two possibilities:\\
a) $L=\langle e\rangle\oplus\langle e^\perp\rangle$\\
b) $[L:\langle e\rangle\oplus\langle e^\perp\rangle]=2$
\end{lemma}
\begin{proof}
Since $e$ -- $2k$-root, $\frac{2e}{2k}\in L^V \ \Rightarrow \ (e,L)=k\mathbb{Z}$. So, the projection of $L$ on $\mathbb{R}e$ consists of multiples of $\frac{k}{2k}e=\frac{e}{2}$. Since $e\in pr_eL$, $[L:\langle e\rangle\oplus\langle e^\perp\rangle]=[pr_eL:\langle e \rangle]=1$ or $2$.
\end{proof}

\begin{lemma}
Automorphism $\phi:\langle e^\perp\rangle \rightarrow \langle e^\perp\rangle$ can be continued to the  automorphism of $L$ $\Leftrightarrow$ $\phi:pr_{e^\perp}L/\langle e^\perp\rangle = id$
\end{lemma}
\begin{proof}
The statement is clear if $L=\langle e\rangle\oplus\langle e^\perp\rangle$. Consider the case $[L:\langle e\rangle\oplus\langle e^\perp\rangle]=2$. Let $x\in L$ be a glueing vector, i.e. such that $(x,e)=k$. Then $L=\mathbb{Z} x\oplus \mathbb{Z}e\oplus e^\perp$.\\
$\Rightarrow$: If $\phi:\langle e^\perp\rangle \rightarrow \langle e^\perp\rangle$ is continued on $L$, then $\phi(e)=\pm e$, and, taking the composition $r_e\circ \phi$ if needed, we can consider only the case $\phi(e)=e$. Then $\phi(pr_e x)=\phi(\frac{e}{2})=\frac{e}{2}=pr_e(\phi(x))$. On the other hand, $\phi(x)\in L \ \Rightarrow \ \phi(x)-x\in L \ \Rightarrow \ \phi(pr_{e^\perp} x)-pr_{e^\perp} x \in \langle e^\perp \rangle \Rightarrow \phi(pr_{e^\perp} x)=pr_{e^\perp} x \pmod{\langle e^\perp \rangle}$. \\
$\Leftarrow$: Let $\phi(pr_{e^\perp} x)=pr_{e^\perp} x \pmod{\langle e^\perp \rangle}$. $L=\langle e \rangle \oplus \langle e^\perp \rangle \cup \frac{e}{2}+pr_{e^\perp} x=\langle e \rangle \oplus \langle e^\perp \rangle \cup \frac{e}{2}+\phi(pr_{e^\perp} x)$. Let $\phi$ act identically on $e$, then under the action of $\phi$ the glueing vector maps to the glueing vector, so $\phi$ is the automorphism of the whole lattice $L$.
\end{proof}

\begin{lemma}
Each automorphism $\phi:\langle e^\perp\rangle \rightarrow \langle e^\perp\rangle$ can be continued to the automorphism of $L$.
\end{lemma}
\begin{proof}
$pr_{e^\perp}L/\langle e^\perp\rangle = pr_{e}L/\langle e \rangle= \mathbb{Z}/2\mathbb{Z}$. But $\mathbb{Z}/2\mathbb{Z}$ has no non-trivial automorphisms.
\end{proof}

\begin{lemma}
a) Let $x$ be a glueing vector for $\langle e \rangle \oplus \langle e^\perp \rangle$. Then $x\ \pmod{\langle e \rangle \oplus \langle e^\perp \rangle}$ is uniquely defined by an element of order 2 in $disc(\langle e^\perp \rangle):= \langle e^\perp \rangle ^V /  \langle e^\perp \rangle$.\\
b) Let $e$ be a primitive root of $L$ such that the class of $\langle e^\perp \rangle$ coincides with the genus. Then the number of roots, conjugated to $e$ in $L$ is equal to the number of glueing vectors $\pmod{\langle e \rangle \oplus \langle e^\perp \rangle}$.
\end{lemma}
\begin{proof}
a) $pr_{e^\perp}x\subseteq \langle e^\perp \rangle ^V$, since $x$ is a glueing vector. Besides, $2pr_{e^\perp}x \in \langle e^\perp \rangle$, because $pr_{e^\perp}L/\langle e^\perp\rangle = \mathbb{Z}/2\mathbb{Z}$. $pr_e x = \frac{e}{2}$ is uniquely defined, so $x=\frac{e}{2}+pr_{e^\perp}x$ is uniquely defined.\\
b) Let $e,\ e'$ be the primitive $-2k$-roots of $L$, $x,\ x'$ -- the corresponding glueing vectors. Because $e^\perp$, $e'^\perp$ belong to the same class, there exists an automorphism, mapping $e$ to $e'$, $e^\perp$ to $e'^\perp$. Taking the composition with reflection $r_e$ if needed, we can think that this automorphism belongs to $O^+(L)$. This automorphism maps $x$ to $x'$ $\Leftrightarrow$ $x$ and $x'$ are defined by the same element in $disc(\langle e^\perp \rangle)=disc(\langle e'^\perp \rangle)$
\end{proof}
\subsection*{Reflections, $d=1 \pmod{4}$}
The invariant factors of $L$ are $1,\ 1,\ 1,\ d$. Each root has an even length, dividing $2d$ (doubled maximal invariant factor). That's why a priori $-2p_{i_1}\ldots p_{i_m}$ reflections can exist, where $\lbrace i_1,\ldots i_m \rbrace \subseteq \lbrace 1,\ldots,k \rbrace$ (including the empty set).

\begin{lemma} 
Let $e$ be a root,  $e^2=-2p_{i_1}\ldots p_{i_m}=:-2e'$. Then 
$det( \langle e^\perp \rangle)= \frac{-2d}{p_{i_1}\ldots p_{i_m}}=:-2d_e$.
\end{lemma}
\begin{proof}
There are two possibilities by lemma 5: $L=\langle e\rangle\oplus\langle e^\perp\rangle$ or
$[L:\langle e\rangle\oplus\langle e^\perp\rangle]=2$. The length of the root is even, but the discriminant of the lattice $L$ is odd, that's why we get a contradiction in the first case.
\end{proof}


\begin{lemma}
If $e$ is is primitive $-2e'$ root of $L$, then it is unique up to conjugation.
\end{lemma}
\begin{proof}
Let $x$ be a glueing vector for $\langle e \rangle \oplus \langle e^\perp \rangle$. It suffices to find $x$ locally because by Kneser's theorem the class of $L$ coincides with the genus. Over each $\mathbb{Z}_p,\ p\neq 2$, 2 is invertible, so $x\in \frac{1}{2}L\otimes \mathbb{Z}_p \  \Rightarrow x\in L\otimes \mathbb{Z}_p$. Over $\mathbb{Z}_2$ $disc(\langle e^\perp \rangle)\cong \mathbb{Z} / 2\mathbb{Z}$. It follows from lemma 8 that $x$ is uniquely defined by $pr_{e^\perp} x $, that corresponds to the unique element of order 2 in $disc(\langle e^\perp \rangle)\cong \mathbb{Z} / 2\mathbb{Z}$. So, there is no more than one glueing vector $x$. By lemma 1 the lattice $\langle e^\perp \rangle$ is unique in it's class. 
It follows from lemma 8.b) that $e$ is unique.
\end{proof}
Let $\{e_i\}$ be the basis for $L$.
\begin{lemma}
For $e^2=-2p_{i_1}\ldots p_{i_m}$ to be a root it is necessary that $\left( \frac{d_e}{p_{i_j}} \right)=1$ for each $p_{i_j},\ 1\leq j \leq m$. 
\end{lemma}
\begin{proof}
Let $e=(x,\ y,\ z,\ t)$ be a root. We can rewrite this in coordinates: $e^2=-2p_{i_1}\ldots p_{i_m}$: $2xy+2z^2+2tz+\frac{1-d}{2}t^2=-2p_{i_1}\cdot \ldots \cdot p_{i_m}$. Besides, $\frac{e}{p_{i_1}\ldots p_{i_m}}\in L^V$. Taking the scalar product $(e,e_i)$, we get that $p_{i_1}\ldots p_{i_m} | x,$ $\ p_{i_1}\ldots p_{i_m}|y,$ $\ p_{i_1}\ldots p_{i_m}|2z+t$. Let $x=p_{i_1}\ldots p_{i_m} x'$, $y=p_{i_1}\ldots p_{i_m}y'$, $2z=p_{i_1}\ldots p_{i_m}t'-t$, substituting this to the first equation, we get:
$$ 2x'y'(p_{i_1}\ldots p_{i_m})^2+\frac{(p_{i_1}\ldots p_{i_m}t'-t)(p_{i_1}\ldots p_{i_m}t'+t)}{2}+\frac{1-d}{2}t^2=-2p_{i_1}\ldots p_{i_m}$$
$$2x'y'(p_{i_1}\ldots p_{i_m})^2+\frac{(p_{i_1}\ldots p_{i_m})^2t'^2-dt^2}{2}=-2p_{i_1}\ldots p_{i_m}$$
Multiplying by 2 and dividing by $p_{i_1}\ldots p_{i_m}$, we get:
$$4x'y'p_{i_1}\ldots p_{i_m}+p_{i_1}\ldots p_{i_m}t'^2-\frac{d}{p_{i_1}\ldots p_{i_m}}t^2=-4$$
$$4x'y'p_{i_1}\ldots p_{i_m}+p_{i_1}\ldots p_{i_m}t'^2-d_e t^2=-4$$

Considering the last equation mod each $p_{i_j}$, we get that $d_e t^2 = 4 \pmod{p_{i_j}}$ must be solvable, i.e. $\left( \frac{d_e}{p_{i_j}} \right)=1$.
\end{proof}

Let $d$ be a prime number ($d=p$). We can find all the roots explicitly in this case:
\begin{lemma}
Let $e$ be a primitive root of $L$. There are two possibilities:\\
a) $e^2=-2$, $\langle e^\perp \rangle= U\oplus \langle 2p \rangle$.\\
b)$e^2=-2p$, $\langle e^\perp \rangle=U\oplus \langle 2 \rangle$.
\end{lemma}
\begin{proof}
a) Let $e=e_1-e_2$. $e^2=-2,\ \frac{2e}{-2}\in L \subseteq L^V$. The Hasse invariants of lattices $\langle e^\perp \rangle$ and $U\oplus \langle 2p \rangle$ for each prime, except $p$ and $2$ are equal to 1, because the lattices are unimodular. It follows from the proof of lemma 1 that over $\mathbb{Z}_2$ the lattice is uniquely defined. We compute the Hasse invariants over $\mathbb{Z}_p$: $$\epsilon_p(U\oplus \langle 2p \rangle)=(-1,\ 2p)_p=\left( \frac{-1}{p}\right)=1,\ \text{ since }\ p=1\pmod{4}$$ $$\left( \frac{-2}{p}\right)=\epsilon_p(L)=\epsilon_p(\langle e\rangle\oplus\langle e^\perp\rangle)=(-2,-2p)_p\epsilon_p(\langle e^\perp \rangle)=\left( \frac{-2}{p}\right)\epsilon_p(\langle e^\perp\rangle).$$ So, $\epsilon_p(\langle e^\perp\rangle)=1$, i.e. the Hasse invariants are equal for all primes. By lemma 1 $\langle e^\perp \rangle = U\oplus \langle 2p \rangle$.

b) Let $e=2e_4-e_3$. $e^2=4\cdot \frac{1-p}{2}-2\cdot 1 - 2\cdot 1 + 2 = -2p$. $\frac{2(2e_4-e_3)}{-2p}\in L^V$, because $(\frac{2e_4-e_3}{p},e_3)=\frac{2-2}{p}=0,\ (\frac{2e_4-e_3}{p},e_4)=\frac{1-p-1}{p}=-1$. The lattices $\langle e^\perp \rangle$ and $U\oplus \langle 2 \rangle$ have the same determinants and Hasse invariants (for $p\neq 2$ they are unimodular, i.e. $\epsilon_p=1$, for $p=2$ the lattie is uniquely defined). Thus, they are equivalent by lemma 1.
\end{proof}

By lemma 10 these roots are unique.

\subsection*{Reflections, $d=3 \pmod{4}$}
The invariant factors are $1,\ 1,\ 2,\ 2d$. Each root has an even length, dividing $4d$ (doubled maximal invariant factor). That's why a priori $-2p_{i_1}\ldots p_{i_m}$ and $-4p_{i_1}\ldots p_{i_m}$ reflections can exist, where $\lbrace i_1,\ldots i_m \rbrace \subseteq \lbrace 1,\ldots,k \rbrace$ (including the empty set).

\begin{lemma} Let $e$ be a primitive root of $L$, then one of the following statements is true:\\
a) $e^2=-4p_{i_1}\ldots p_{i_m}=:-4e'$,
$det(\langle e^\perp\rangle) = \frac{-4d}{p_{i_1}\ldots p_{i_m}}=:-4d_e$.\\
b) $e^2=-2p_{i_1}\ldots p_{i_m}=:-2e'$,
$det(\langle e^\perp \rangle) = \frac{-2d}{p_{i_1}\ldots p_{i_m}}=:-2d_e$.\\
c)$e^2=-2p_{i_1}\ldots p_{i_m}=:-2e'$, $det(\langle e^\perp \rangle) = \frac{-8d}{p_{i_1}\ldots p_{i_m}}=:-8d_e$.
\end{lemma}
\begin{proof}
The statements b), c) correspond to the cases from lemma 5.a) and 5.b) respectively for $e^2=-2p_{i_1}\ldots p_{i_m}$.\\
For $e^2=-4p_{i_1}\ldots p_{i_m}$ the case a) from lemma 5 is impossible: it follows from the equality $e^2 \cdot det(\langle e^\perp \rangle)=4d$ that $det(\langle e^\perp \rangle)$ is odd, but it contradicts that $\langle e^\perp \rangle$ is even lattice of rank 3. Thus, $[L:\langle e \rangle\oplus\langle e^\perp \rangle]=2$. It corresponds to the case $det(\langle e^\perp \rangle) = \frac{-8d}{p_{i_1}\ldots p_{i_m}}$.
\end{proof}

\begin{lemma}
Let $e$ be a primitive root of $L$ such that $e^2=-2e',\ det(\langle e^\perp\rangle)=-8d_e$, then: \\
a) $\epsilon_2(\langle e^\perp\rangle)=-1$\\
b) The invariant factors of $\langle e^\perp\rangle$ are $2,\ 2,\ 2d_e$\\
c) If $e'=3 \pmod{4}$, then over $\mathbb{Z}_2$ $\langle e^\perp \rangle \sim \begin{pmatrix}
4 & 2  \\
2 & 4 \\
\end{pmatrix} \oplus \langle -6d_e \rangle$.
\end{lemma}
\begin{proof}
a) $\epsilon_2(L)=-(2,2d)=\begin{cases} 
1,&\text{if $d=3 \pmod{8}$;}\\
-1,&\text{if $d=-1 \pmod{8}$;}
\end{cases}$. On the other hand, $\epsilon_2(L)=\epsilon_2(\langle e \rangle \oplus \langle e^\perp\rangle)=\epsilon_2(\langle e^\perp \rangle)\cdot (-8d_e,-2e')_2$. Going through all the options on $d_e, \ e' \pmod{8}$, we obtain the statement.\\

b) Assume the invariant factors of $\langle e^\perp \rangle$ are not $2,\ 2,\ 2d_e$, then they must be $1,\ 2,\ 4d_e$ or $1,\ 1,\ 8d_e$. But the first possibility is impossible $1,\ 2,\ 4d_e$, because the lattice $\langle e^\perp \rangle$ is even, and so it can't have an unimodular component of rank 1. Assume the invariant factors are $1,\ 1,\ 8d_e$. Let $x$ be the glueing vector for $\langle e\rangle\oplus\langle e^\perp \rangle$. It suffices to find $x$ locally because the class of $L$ coincides with the genus by Kneser's theorem. Over each $\mathbb{Z}_p,\ p\neq 2$, 2 is invertible, so $x\in \frac{1}{2}L\otimes \mathbb{Z}_p \  \Rightarrow x\in L\otimes \mathbb{Z}_p$. Over $\mathbb{Z}_2$ $disc(\langle e^\perp \rangle)\cong \mathbb{Z} / 8\mathbb{Z}$. It follows from lemma 8 that $x$ is uniquely defined by $pr_{e^\perp} x $, that corresponds to the unique element of order 2 in $disc(\langle e^\perp \rangle)\cong \mathbb{Z} / 8\mathbb{Z}$. So, there is no more than one glueing vector $x$. But $x^2=\left( \frac{1}{2}e \right)^2 + ( pr_{e^\perp}x )^2 = \frac{-2e'}{4} + \frac{1}{4}\cdot 8d_e = \frac{e'}{2} \neq 0 \pmod{\mathbb{Z}}$, thus $x$ is not a glueing vector. So there is no glueing vector and this case is impossible.\\

c) It follows from b) that we can apply lemma 3 for $\langle e^\perp \rangle$. We notice that $e'=3 \pmod{4} \Rightarrow d_e=1 \pmod{4}$, because $e'\cdot d_e=d=3\pmod 4$. Taking into account that $\epsilon_2(\langle e^\perp\rangle)=-1$, the required statement follows from the proof of lemma 3.
\end{proof}

It follows from lemmas 1, 2, 3, 14 that in every case listed in lemma 13, the class of $\langle e^\perp \rangle$ coincides with the genus.
 
\begin{lemma}
a) If $e$ is a primitive $-4p_{i_1}\ldots p_{i_m}$ root of $L$, then it is unique up to conjugation.\\
b) There can be at most two primitive $-2p_{i_1}\ldots p_{i_m}$ roots $e_1,\ e_2$ in $L$. In this case  $det(\langle e_1^\perp \rangle)=-2d_{e_1}$, $det(\langle e_2^\perp \rangle)=-8d_{e_2}$. 
\end{lemma}
\begin{proof}
a) Let $x$ be the glueing vector for $\langle e \rangle\oplus\langle e^\perp\rangle$. It suffices to find $x$ locally because the class of $L$ coincides with the genus by Kneser's theorem. Over each $\mathbb{Z}_p,\ p\neq 2$, 2 is invertible, so $x\in \frac{1}{2}L\otimes \mathbb{Z}_p \  \Rightarrow x\in L\otimes \mathbb{Z}_p$. Over $\mathbb{Z}_2$ $disc(\langle e^\perp \rangle)\cong \mathbb{Z} / 4\mathbb{Z}$. It follows from lemma 8 that $x$ is uniquely defined by $pr_{e^\perp} x $, that corresponds to the unique element of order 2 in $disc(\langle e^\perp \rangle)\cong \mathbb{Z} / 4\mathbb{Z}$. So
there is no more than one glueing vector $x$. By lemma 2 the lattice $\langle e^\perp \rangle$ is unique in it's class. It follows from lemma 8.b) that $e$ is unique up to conjugation.

b) The uniqueness of $-2p_{i_1}\ldots p_{i_m}$ root $e_1$ such that $det(\langle e_1^\perp\rangle)=-2d_{e_1}$ follows from lemma 1 and the fact that $e_1$ splits off as a direct summand.\\
Let $x$ be a glueing vector for $\langle e_2 \rangle\oplus \langle e_2^\perp \rangle$. It suffices to find $x$ locally because the class of $L$ coincides with the genus by Kneser's theorem. Over each $\mathbb{Z}_p,\ p\neq 2$, 2 is invertible, so $x\in \frac{1}{2}L\otimes \mathbb{Z}_p \  \Rightarrow x\in L\otimes \mathbb{Z}_p$. Over $\mathbb{Z}_2$ $disc(\langle e_2^\perp \rangle)\cong \left( \mathbb{Z} / 2\mathbb{Z}\right)^3$, because by lemma 14.b) the invariant factors are $2,\ 2,\ 2d_e$. It means that there are 7 elements of order 2 in $disc(\langle e^\perp \rangle)$. From the proof of lemma 3 we know the explicit form of $\langle e_2^\perp\rangle$: $d_e=-1 \pmod{4} \Rightarrow \ \langle e^\perp \rangle \sim U(2)\oplus \langle 2d_e \rangle$, $d_e=1 \pmod{4} \Rightarrow \langle e^\perp\rangle\sim\begin{pmatrix}
4 & 2  \\
2 & 4 \\
\end{pmatrix} \oplus \langle -6d_e \rangle$. The length of the glueing vectors must be even so going through all the possibilities we get that there is at most one glueing vector: $d_e=-1 \pmod{4} \Rightarrow x=(1,1,1)\in \left( \mathbb{Z} / 2\mathbb{Z}\right)^3 $, $d_e=1 \pmod{4} \Rightarrow x=(0,0,1)\in \left( \mathbb{Z} / 2\mathbb{Z}\right)^3 $. 
By lemma 3 the lattice $\langle e^\perp \rangle$ is unique in it's class. It follows from lemma 8.b) that $e$ is unique up to conjugation.
\end{proof}

\begin{lemma}
$e=-2e'$ is a primitive root of $L$, $det \langle e^\perp\rangle=-2d_e$ $\Rightarrow$ $e'=1 \pmod{4}$
\end{lemma}
\begin{proof}
In the Conway-Sloane notation the lattice $L$ has the form $1_{II}^{+2}2_{d-1}^{\omega(d)2}$. The lattice $U\oplus \langle 2d_e \rangle \oplus \langle -2e'\rangle$ has the form $1_{II}^{+2}2_{d_e-e'}^{\omega(d_e)\omega(e')2}$. Going through all the options on $e', \ d_e \ \pmod{8}$, we obtain the equivalence of these lattices $\Leftrightarrow \ e'=1 \pmod{4}$.
\end{proof}

Let $e_i$ be the basis of $L$.
\begin{lemma}
For $e^2=-2kp_{i_1}\ldots p_{i_m}$ to be a root ($k=1$ or $k=2$) it is necessary that $d_e t^2 = -k \pmod{p_{i_j}}$ is solvable for each $p_{i_j},\ 1\leq j \leq m$.
\end{lemma}
\begin{proof}
 Let $e=(x,\ y,\ z,\ t)$ be a root. We can rewrite this in coordinates: $e^2=-2kp_{i_1}\ldots p_{i_m}$: $2xy-2z^2+2dt^2=-2kp_{i_1}\cdot \ldots \cdot p_{i_m}$. Besides, $\frac{e}{p_{i_1}\ldots p_{i_m}}\in L^V$. Taking the scalar product $(e,e_i)$, we get that $p_{i_1}\ldots p_{i_m} | x,$ $ p_{i_1}\ldots p_{i_m}|y,$ $ p_{i_1}\ldots p_{i_m}|z$. Let $x=p_{i_1}\ldots p_{i_m} x'$, $y=p_{i_1}\ldots p_{i_m}y'$, $z=p_{i_1}\ldots p_{i_m}z'$, substituting this to the first equation, we get:
$$ 2x'y'(p_{i_1}\ldots p_{i_m})^2-2(p_{i_1}\ldots p_{i_m})^2 z'^2+2dt^2=-2kp_{i_1}\ldots p_{i_m}$$
Dividing by $2p_{i_1}\ldots p_{i_m}$, we get:
$$x'y'p_{i_1}\ldots p_{i_m}-p_{i_1}\ldots p_{i_m}z'^2+\frac{d}{p_{i_1}\ldots p_{i_m}}t^2=-k$$
$$x'y'p_{i_1}\ldots p_{i_m}-p_{i_1}\ldots p_{i_m}z'^2+d_e t^2=-k$$

Considering the last equality mod each $p_{i_j}$, we get that $d_e t^2 = -k \pmod{p_{i_j}}$ must be solvable.

\end{proof}
\subsection*{Reflections, $d=2 \pmod{4}$}
The invariant factors of $L$ are $1,\ 1,\ 2,\ 4p_1\ldots p_k=:4d'$. Each root has an even length, dividing $8d'$ (doubled maximal invariant factor). That's why a priori $-2p_{i_1}\ldots p_{i_m}$, $-4p_{i_1}\ldots p_{i_m}$ and $-8p_{i_1}\ldots p_{i_m}$-roots can exist, where $\lbrace i_1,\ldots i_m \rbrace \subseteq \lbrace 1,\ldots,k \rbrace$ (including the empty set).
Let $e_i$ be the basis of $L$.
\begin{lemma}
There are no $-8p_{i_1}\ldots p_{i_m}$-roots.
\end{lemma}
\begin{proof}
Let $e=(x,\ y,\ z,\ t)$ be a primitive $-8p_{i_1}\ldots p_{i_m}$-root.  We can rewrite this in coordinates: $2xy-2z^2+4d't^2=-8p_{i_1}\cdot \ldots \cdot p_{i_m}$. Besides, $\frac{e}{4p_{i_1}\ldots p_{i_m}}\in L^V$. Taking the scalar product $(e,e_i)$, we get that $4p_{i_1}\ldots p_{i_m} | x,$ $ 4p_{i_1}\ldots p_{i_m}|y,$ $ 2p_{i_1}\ldots p_{i_m}|z$. Let $x=4p_{i_1}\ldots p_{i_m} x'$, $y=4p_{i_1}\ldots p_{i_m}y'$, $z=2p_{i_1}\ldots p_{i_m}z'$, substituting this to the first equation, we get:
$$ 16x'y'(p_{i_1}\ldots p_{i_m})^2-4(p_{i_1}\ldots p_{i_m})^2 z'^2+2d't^2=-4p_{i_1}\ldots p_{i_m}$$
Dividing by $2p_{i_1}\ldots p_{i_m}$, we get:
$$8x'y'p_{i_1}\ldots p_{i_m}-2p_{i_1}\ldots p_{i_m}z'^2+\frac{d'}{p_{i_1}\ldots p_{i_m}}t^2=-2$$
We obtain that $2|t$, and that contradicts that $e$ is primitive. 
\end{proof}

\begin{lemma} Let $e$ be a primitive root of $L$, then one of the following statements is true:\\
a) $e^2=-4p_{i_1}\ldots p_{i_m}=:-4e'$,
$det(\langle e^\perp\rangle) = \frac{-2d'}{p_{i_1}\ldots p_{i_m}}=:-2d_e$.\\
b) $e^2=-4p_{i_1}\ldots p_{i_m}=:-4e'$,
$det(\langle e^\perp\rangle) = \frac{-8d'}{p_{i_1}\ldots p_{i_m}}=:-8d_e$.\\
c) $e^2=-2p_{i_1}\ldots p_{i_m}=:-2e'$,
$det(\langle e^\perp \rangle) = \frac{-4d'}{p_{i_1}\ldots p_{i_m}}=:-4d_e$.\\
d)$e^2=-2p_{i_1}\ldots p_{i_m}=:-2e'$, $det(\langle e^\perp \rangle) = \frac{-16d'}{p_{i_1}\ldots p_{i_m}}=:-16d_e$.
\end{lemma}
\begin{proof}
The statements a), b) corresponds to the cases from lemma 5.a) and 5.b) respectively for $e^2=-4p_{i_1}\ldots p_{i_m}$, the statements c) and d) correspond to the cases from lemma 5.a) and 5.b) respectively for $e^2=-2p_{i_1}\ldots p_{i_m}$.
\end{proof}

\begin{lemma}
a) If $e$ is a primitive $-4p_{i_1}\ldots p_{i_m}$-roots of $L$ such that $det(\langle e^\perp \rangle) =-2d_e$, then it is unique up to conjugacy.\\
b) There exist at most three different $e$ -- primitive $-4p_{i_1}\ldots p_{i_m}$-roots of $L$ such that $det(\langle e^\perp \rangle) =-8d_e$ up to conjugacy. Moreover, if $\langle e^\perp \rangle \cong U(2)\oplus \langle 2d_e\rangle$, then $e$ is unique up to conjugacy.\\
c) If $e$ is a primitive $-2p_{i_1}\ldots p_{i_m}$-root of $L$ such that $det(\langle e^\perp \rangle) =-4d_e$, then it is unique up to conjugacy.\\
d) There exist at most two different $e$ -- primitive $-2p_{i_1}\ldots p_{i_m}$-roots of $L$ such that $det(\langle e^\perp \rangle) =-16d_e$ up to conjugacy.
\end{lemma}
\begin{proof}
a) It follows from lemma 1 and the fact that $e$ splits off as a direct summand.\\
b) The invariant factors can't be 1,1,8. Otherwise $disc(e^\perp)\cong \mathbb{Z}/8\mathbb{Z}$, and there exists a unique element of order 2. If $x$ is the glueing vector then $x^2=(\frac{1}{2}e)^2+(pr_{e^\perp}x)^2=\frac{1}{4}e^2+\frac{1}{4}8d_e$, i.e. the glueing vector has an odd length and that is a contradiction. By lemma 3 $\langle e^\perp \rangle$ is unique in it's genus. $disc(e^\perp)\cong (\mathbb{Z}/2\mathbb{Z})^3$. Going through all the elements of order 2 and the corresponding glueing vectors we get that if $\langle e^\perp \rangle \cong \begin{pmatrix}
4 & 2  \\
2 & 4 \\
\end{pmatrix} \oplus \langle -6d' \rangle$, then the only glueing vectors that have an integer even length are those corresponding to the elements (1,0,0), (0,1,0) and (1,1,0). If $\langle e^\perp \rangle \cong U(2)\oplus \langle 2d_e\rangle$, then the only glueing vector that has an integer even length is the one corresponding to (1,1,0) or (1,1,1) (their lengths have different oddities so only one of them is possible). The statement follows from lemma 8.\\
c) It follows from lemma 2 and the fact that $e$ splits off as a direct summand.\\
d) The invariant factors can't be $1,1,16d_e$, because then $disc(e^\perp)\cong \mathbb{Z}/16\mathbb{Z}$, and there exists a unique element of order 2. If $x$ is a glueing vector then $x^2=(\frac{1}{2}e)^2+(pr_{e^\perp}x)^2=\frac{1}{4}e^2+\frac{1}{4}16d_e$, i.e. the length of the glueing vector is not integer, which is impossible. So, the invariant factors are $2,2,4d_e$, $disc(\langle e^\perp \rangle)\cong \mathbb{Z}/2\mathbb{Z}\times \mathbb{Z}/2\mathbb{Z}\times \mathbb{Z}/4\mathbb{Z}$. There are 7 elements of order 2 in this group. By lemma 4 $\langle e^\perp \rangle\cong U(2)\oplus \langle 4d_e \rangle$ or $\langle e^\perp \rangle\cong \langle 2\epsilon_1 \rangle \oplus \langle 2\epsilon_2 \rangle \oplus \langle -4\epsilon_1 \epsilon_2 \rangle$. 

In the first case we get that all elements of order 2 in $disc(\langle e^\perp \rangle)$ correspond to a glueing vector with non-integer length. So, this case is impossible. 

All the possibilities for $\langle e^\perp \rangle$ are listed in the table in lemma 4. Going through all the possibilities we get that there are at most 2 different glueing vectors of even length in each case. The statement follows from lemma 8.

\end{proof}

\begin{lemma}
For $e^2=-2kp_{i_1}\ldots p_{i_m}$ to be a root ($k=1$ or $k=2$), it is necessary  that $2d_e t^2 = -k \pmod{p_{i_j}}$ is solvable for each  $p_{i_j},\ 1\leq j \leq m$.
\end{lemma}
\begin{proof}
Let $e=(x,\ y,\ z,\ t)$ be a root.  We can rewrite this in coordinates: $e^2=-2kp_{i_1}\ldots p_{i_m}$: $2xy-2z^2+4d't^2=-2kp_{i_1}\cdot \ldots \cdot p_{i_m}$. Besides, $\frac{e}{p_{i_1}\ldots p_{i_m}}\in L^V$. Taking the scalar product $(e,e_i)$, we get that $p_{i_1}\ldots p_{i_m} | x,$ $p_{i_1}\ldots p_{i_m}|y,$ $p_{i_1}\ldots p_{i_m}|z$. Let $x=p_{i_1}\ldots p_{i_m} x'$, $y=p_{i_1}\ldots p_{i_m}y'$, $z=p_{i_1}\ldots p_{i_m}z'$, substituting this to the first equation, we get:
$$ 2x'y'(p_{i_1}\ldots p_{i_m})^2-2(p_{i_1}\ldots p_{i_m})^2 z'^2+4d't^2=-2kp_{i_1}\ldots p_{i_m}$$
Dividing by $2p_{i_1}\ldots p_{i_m}$, we get:
$$x'y'p_{i_1}\ldots p_{i_m}-p_{i_1}\ldots p_{i_m}z'^2+2\frac{d'}{p_{i_1}\ldots p_{i_m}}t^2=-k$$
$$x'y'p_{i_1}\ldots p_{i_m}-p_{i_1}\ldots p_{i_m}z'^2+2d_e t^2=-k$$

Considering it mod each $p_{i_j}$ we get that $2d_e t^2 = -k \pmod{p_{i_j}}$ must be solvable. 

\end{proof}

\section*{ Estimation of volumes, d=1 mod 4}



\begin{lemma}
Let e be a primitive root of $L$,  $-2d_e=-2p_{i_1}\ldots p_{i_m}$. The maximal possible volume $Vol(\langle e^\perp \rangle)=\frac{1}{24}\frac{p_{i_1}+1}{2} \cdot \ldots \cdot \frac{p_{i_m}+1}{2}$
\end{lemma}
\begin{proof}
The algorithm for computing the volume can be found in the article \cite{gri}, \cite{sie} and the book \cite{kit}. One needs to compute the local densities $a_p$ of the lattice $L$ for each prime $p$. After that the volume of the lattice $L$ with signature $(2,n)$ is computed by the formula 
\begin{equation*}
V=2|det(L)|^{\frac{n+3}{2}}\prod\limits_{k=1}^{n+2}\pi^{-k/2}\Gamma\left(\frac{k}{2}\right)\cdot \prod a_p^{-1}
\end{equation*}
For $p=2$ $a_2=2^4(1-2^{-2})$. For each $p\neq 2$ and $p\neq p_i$ ($p_i$ -- divisor of $det\langle e^\perp \rangle $), $a_p=(1-p^{-2})$, because the lattice is unimodular. For $p=p_i$ $a_{p_i}=2p_i(1-p_i^{-2})(1\pm p_i^{-1})^{-1}$. The sign "plus" is chosen when the unimodular component is equivalent to the hyperbolic plane, otherwise the sign "minus" is chosen. To compute the volume we multiply by the inverse of each local density $a_p$. So, the maximum is reached when all the signs are pluses. It corresponds to the lattice $U\oplus \langle -2d_l \rangle$, the volume of which is in the statement of the lemma. 
\end{proof}

\begin{lemma}
$Vol(O(L)) = \frac{d^{3/2}}{2^{k+5}\cdot 3\pi^2} L(2,d)$.
\end{lemma}
\begin{proof}
The local densities are:\\
$a_2=2^4(1-2^{-2})(1-\left( \frac{d}{2}\right) 2^{-2})$\\
$a_{p_i}=2p_i (1-p_i^{-2})$, where $p_i$ -- divisor of $d$\\
$a_p=(1-p^{-2})(1-\left( \frac{d}{2}\right) p^{-2})$\\
Substituting it in the volume formula we get the statement of the lemma.
\end{proof}

Let $d$ be a prime ($d=p$). All the existing roots and the explicit form of their orthogonal complements are stated in lemma 12. So their volumes can also be calculated explicitly:
\begin{lemma}
$Vol(\langle -2p^\perp \rangle)=\frac{1}{24}$, $Vol(\langle -2^\perp \rangle)=\frac{1}{24}\frac{p+1}{2}$. 
\end{lemma}
\begin{proof}
By lemma 12 $\langle -2p^\perp \rangle = U\oplus \langle 2 \rangle$. For $p\neq 2 $ the lattice is unimodular, $a_p=(1-p^{-2})$. For $p=2$ $a_2=2^4(1-2^{-2})$. Substituting it in the volume formula we get the first statement of the lemma.\\
By lemma 12 $\langle -2^\perp \rangle = U\oplus \langle 2p \rangle$. For $p'\neq 2,\ p'\neq p $ the lattice is unimodular, $a_{p'}=(1-p'^{-2})$. For $p=2$ $a_2=2^4(1-2^{-2})$. $a_{p'}=2p'(1-p'^{-2})(1+ p'^{-1})^{-1}$. The sign "plus" in the last bracket is chosen because the unimodular component is equivalent to the hyperbolic plane. Substituting it in the volume formula, we get the second statement.
\end{proof}

\subsection*{The proof of theorem 3}
\begin{proof}

We need to find the lower bound for the left hand side of formula (1). $K\geq 8$, $L(s,\chi)=\prod (1-\frac{\chi(p)}{p^s})^{-1}\geq \prod (1+p^{-s})^{-1}=\prod (1-p^{-2s})^{-1}(1-p^{-s})=\zeta(2s)/\zeta(s)$. The first inequality follows from the fact that in each multiple $\chi(p)\in \lbrace \pm 1,\ 0\rbrace$ and  the maximum is reached when $\chi(p)=-1$. So $\frac{d^{3/2}}{2^{k+5}\cdot 3\pi^2} L(2,d)\cdot K  \geq \frac{d^{3/2}}{2^{k+5}\cdot 3\pi^2} \frac{\zeta_Q(4)}{\zeta_Q(2)} \cdot 8=\frac{d^{3/2}}{2^{k+2}\cdot 3^2\cdot 5}$. 
 
Now we need to find the upper bound for the right hand side of formula (1). $Vol(div(F))\leq \frac{1}{2} \frac{1}{24} \sum \frac{p_{i_1}+1}{2}\cdot \ldots \cdot \frac{p_{i_m}+1}{2}$, where the sum is taken over all subsets $\lbrace i_1,\ldots,i_m \rbrace \subseteq \lbrace 1,\ldots, k\rbrace$. We notice that this sum is equal to $(\frac{p_1+1}{2}+1)\cdot \ldots \cdot (\frac{p_k+1}{2}+1)=\frac{1}{2^k}(p_1+3)\cdot \ldots \cdot (p_k+3)$.

So we need to compare $(p_1\cdot \ldots \cdot p_k)^{3/2}$ and $\frac{15}{4}(p_1+3)\cdot \ldots \cdot (p_k+3)$.

Consider the set of k-tuples of primes such that their product is equivalent to 1 mod 4, arranged in the ascending order. We equip this set with a componentwise order. 
The right hand side increases faster than the left hand side as $k$ grows. If the l.h.s. is greater than the r.h.s. for each $p_i$ and for $k=k_0$, then it is true for $k>k_0$. We notice that if for $p_1,\ldots p_m$ the l.h.s. is greater than the r.h.s., then substituting $p_i$ with $p_j>p_i$ preserves the inequality.

Assume $k=4$. The minimal 4-tuple is $3,\ 5,\ 7,\ 13$. The l.h.s. is greater than the r.h.s. for this tuple, so each discriminant $d$, containing at least four prime factors, can't correspond to the free algebra of modular forms. 

Assume $k=3$. For 3-tuples $5,\ 7,\ 11$; $3,\ 7,\ 13$ and 3-tuples, majorizing them, the left hand side is greater than the right hand side. The rest 3-tuples have the form $3,\ 5,\ x$. If $x\geq 19$ then the left hand side is greater than the right hand side. So, when $k=3$ the remaining cases are
$$ d=3\cdot 5 \cdot 7,\ d=3\cdot 5 \cdot 11.$$

Assume $k=2$. For 2-tuples $7,\ 11$; $5,\ 13$ the left hand side is greater than the right hand side. The rest 2-tuples have the form $3,\ x$, and if $x\geq 31$ then the left hand side is greater than right hand side. So, when $k=2$ the remaining cases are
$$ d=3 \cdot 7,\ d=3\cdot 11,\ d=3\cdot 19,\ d=3\cdot 23.$$ 

Assume $k=1$. By lemma 19 the right hand side of formula (1) is computed explicitly. If $p\geq 29$ the l.h.s. is greater than the r.h.s. assuming $K\geq 8$. To simplify the computations we transform the l.h.s. of (1). Since $L(2, p)=\frac{\zeta_{\mathbb{Q}(\sqrt{p})}(2)}{\zeta(2)}$, $\zeta(2)=\frac{\pi^2}{6}$, $\zeta_{\mathbb{Q}(\sqrt{p})}(2) = \frac{4\pi^4\cdot \zeta_{\mathbb{Q}(\sqrt{p})}(-1)}{p^{1.5}}$ (functional equation), we get: $ \frac{p^{3/2}}{2^{1+5}\cdot 3\pi^2} L(2,p)=\frac{p^{3/2}\zeta_{\mathbb{Q}(\sqrt{p})}(2)}{2^{1+5}\cdot 3\pi^2 \zeta(2)}=\frac{\zeta_{\mathbb{Q}(\sqrt{p})}(-1)}{8}$ . Thus, $\frac{\zeta_{\mathbb{Q}(\sqrt{p})}(-1)}{8} \cdot K = \frac{1}{2}\cdot \frac{1}{24}\frac{p+3}{2}$, so $K=\frac{p+3}{12\cdot \zeta_{\mathbb{Q}(\sqrt{p})}(-1)}$.

Going through the rest possibilities with sage we get:
\begin{gather*}
\begin{split}
&P = Primes()\\
&x = P.first()\\
&while\ x<30:\\
&\ \ \ \ K.<a> = QuadraticField(x)\\
&\ \ \ \ Z = K.zeta\_function()\\
&\ \ \ \ if x\%4==1:\\
&\ \ \ \ \ \ \ \ print\ x,\ (x+3)/(12*Z(-1))\\
&\ \ \ \ x=P.next(x)\\
\end{split}
\end{gather*}
This leaves only two possibilities with integer $K\geq 8$:
$$d=5, \ K=20; \ d=13,\ K=8$$

Lemma 11 means that our estimates can be improved:\\

Assume $d=3\cdot 7$, then there is no root $e$ such that $\ e'=7,\ d_e=3 \text{ because } \left( \frac{3}{7}\right) = (3)^{\frac{7-1}{2}}=-1$. Since we considered the volume of $\langle e^\perp \rangle$ in the r.h.s. of (1), it should be reduced by $$\ Vol(\langle e^\perp \rangle)=\frac{1}{24}\cdot \frac{7+1}{2}=\frac{1}{6};$$

Assume $d=3\cdot 11$, then there is no root $e$ such that  $\ e'=3,\ d_e=11 \text{ because } \left( \frac{11}{3}\right) = -1$. Since we considered the volume of $\langle e^\perp \rangle$ in the r.h.s. of (1), it should be reduced by $$Vol(\langle e^\perp \rangle)=\frac{1}{24}\cdot \frac{3+1}{2}=\frac{1}{12};$$

Assume $d=3\cdot 19$, then there is no root $e$ such that  $\ e'=19,\ d_e=3 \text{ because } \left( \frac{3}{19} \right)=(3)^{\frac{19-1}{2}}=-1$. Since we considered the volume of $\langle e^\perp \rangle$ in the r.h.s. of (1), it should be reduced by $$Vol(\langle e^\perp \rangle)=\frac{1}{24}\frac{19+1}{2}=\frac{5}{12};$$

Assume $d=3\cdot 23$, then there is no root $e$ such that $\ e'=3,\ d_e=23 \text{ because } \left( \frac{23}{3}\right) = -1.$ Since we considered the volume of $\langle e^\perp \rangle$ in the r.h.s. of (1),   it should be reduced by $$Vol(\langle e^\perp \rangle)=\frac{1}{24}\cdot \frac{3+1}{2}=\frac{1}{12};$$

Assume $d=3\cdot 5 \cdot 7$, then there are no roots such that $\ e'=3,\ d_{e'}=35;\ e''=3\cdot 5,\ d_{e''}=7; \ e'''=3\cdot 7,\ d_{e'''}=5;\ e''''=5\cdot 7,\ d_{e''''}=3.$ $e'$ doesn't exist, because $\left( \frac{35}{3}\right) = -1$. $e''$ doesn't exist, because $\left( \frac{7}{5}\right) = -1.$ $e'''$ doesn't exist, because $\left( \frac{5}{3}\right) = -1$. $e''''$ doesn't exist, because $\left( \frac{3}{5}\right) = -1$. Since we considered the corresponding volumes in the r.h.s. of (1), it should be reduced by
 $$Vol(\langle e'^\perp \rangle)+Vol(\langle e''^\perp \rangle)+Vol(\langle e'''^\perp\rangle)+Vol(\langle e''''^\perp\rangle)=$$ $$=\frac{1}{24}\left( \frac{3+1}{2}+\frac{3+1}{2}\cdot \frac{5+1}{2}+ \frac{3+1}{2}\cdot \frac{7+1}{2}+\frac{5+1}{2}\frac{7+1}{2}\right)=\frac{7}{6};$$

Assume $d=3\cdot 5\cdot 11$, then there are no roots such that  $\ e'=5,\ d_{e'}=33; \ e''=3\cdot 5,\ d_{e''}=11; \ e'''=3 \cdot 11,\ d_{e'''}=5;\ e''''=5\cdot 11,\ d_{e''''}=3$. $e'$ doesn't exist, because $\left( \frac{33}{5}\right) = -1$. $e''$ doesn't exist, because $\left( \frac{11}{3}\right) = -1$. $e'''$ doesn't exist, because $\left( \frac{5}{3}\right) = -1$. $e''''$ doesn't exist, because $\left( \frac{3}{5}\right) = -1$. Since we considered the corresponding volumes in the r.h.s. of (1), it should be reduced by
$$Vol(\langle e'^\perp\rangle)+Vol(\langle e''^\perp\rangle)+Vol(\langle e'''^\perp\rangle)+Vol(\langle e''''^\perp\rangle)=$$ $$=\frac{1}{24}\left( \frac{5+1}{2}+\frac{3+1}{2}\cdot \frac{5+1}{2}+ \frac{3+1}{2}\cdot \frac{11+1}{2}+\frac{5+1}{2}\cdot \frac{11+1}{2}\right)=\frac{13}{8}.$$

We obtain that for all $d=1\pmod{4}$ except for $d=3\cdot 7$ and $d=3 \cdot 11$ the l.h.s. is greater than the r.h.s., thus the corresponding algebras of modular forms can't be free.

Consider these cases in more details. We want to find the explicit form of $\langle e^\perp \rangle \otimes \mathbb{Z}_p$ for all roots $e$ of $L$. We need to compute Hasse invariant of $\langle e^\perp \rangle \otimes \mathbb{Z}_p$, using that $\epsilon_p(L)=\epsilon_p(\langle e^\perp\rangle)\cdot (det \langle e^\perp \rangle, e^2)_p$. By lemma 1 $\langle e^\perp\rangle \sim U\oplus \langle 2d_e\rangle$ over $\mathbb{Z}_p$ iff they have the same Hasse invariants. It also follows from lemma 1 that this is always true over $\mathbb{Z}_2$, that's why the local density $a_2(e^\perp)$ coincides with the one computed in lemma 18. It follows from the proof of lemma 18 that for the rest $p$ $a_p=2p(1-p^{-2})(1+p)^{-1}$ $\Leftrightarrow$ $\langle e^\perp\rangle \sim U\oplus \langle 2d_e\rangle$ over $\mathbb{Z}_p$, otherwise $a_p=2p(1-p^{-2})(1-p)^{-1}$.\\

Assume $D=33$. We notice that $\epsilon_p(L)=1 \Leftrightarrow p\neq 2$. Then $\epsilon_p(\langle -2^\perp \rangle)=\epsilon_p(\langle -2\cdot 3 \cdot 11^\perp \rangle)=1$ for all $p$, $\epsilon_p(\langle -2\cdot 11^\perp \rangle)=1 \Leftrightarrow p>3$. We get that $\langle -2 \cdot 11^\perp \rangle \sim U\oplus \langle 2\cdot 3\rangle $, $\langle -2 \cdot 3 \cdot 11^\perp \rangle \sim U\oplus \langle 2\rangle $ over $\mathbb{Z}_p$ for all $p$, $\langle -2^\perp \rangle = U\oplus \langle 2\cdot 3\cdot 11\rangle \Leftrightarrow p\neq 3, \ p\neq 11 $. Then $Vol(\langle -2\cdot 3 \cdot 11\rangle^\perp)=\frac{1}{24}$, $Vol(\langle -2 \cdot 11\rangle^\perp)=\frac{1}{24}\cdot \frac{3+1}{2}=\frac{2}{24}$, $Vol(\langle -2 \rangle^\perp)=\frac{1}{24}\cdot \frac{3-1}{2}\cdot \frac{11-1}{2}=\frac{5}{24}$. So the r.h.s. of formula (1) is equal to $\frac{1+2+5}{48}=\frac{1}{6}$. We obtain that the lower bound for the l.h.s. of formula (1) $\frac{33\sqrt{33}}{2^4\cdot 3^2 \cdot 5}>\frac{1}{6}$, so $D=33$ the corresponding algebra of modular forms is not free.\\

Assume $D=21$. We notice that $\epsilon_p(L)=1 \Leftrightarrow p\neq 7$. Then $\epsilon_p(\langle -2^\perp \rangle)=\epsilon_p(\langle -2\cdot 3 \cdot 7^\perp \rangle)=1$ for all $p$, $\epsilon_p(\langle -2\cdot 3^\perp \rangle)=1 \Leftrightarrow p\neq 2, \ p\neq 7$. We get that $\langle -2 \cdot 3^\perp \rangle \sim U\oplus \langle 2\cdot 7\rangle $, $\langle -2 \cdot 3 \cdot 7^\perp \rangle = U\oplus \langle 2\rangle $ over $\mathbb{Z}_p$ for all $p$, $\langle -2^\perp \rangle \sim U\oplus \langle 2\cdot 3\cdot 7\rangle \Leftrightarrow p\neq 3, \ p\neq 7 $. Then $Vol(\langle -2\cdot 3 \cdot 7\rangle^\perp)=\frac{1}{24}$, $Vol(\langle -2 \cdot 3\rangle^\perp)=\frac{1}{24}\cdot \frac{7+1}{2}=\frac{4}{24}$, $Vol(\langle -2 \rangle^\perp)=\frac{1}{24}\cdot \frac{3-1}{2}\cdot \frac{7-1}{2}=\frac{3}{24}$. So the r.h.s. of formula (1) is equal to $\frac{1+4+3}{48}=\frac{1}{6}$. It follows from formula (1) that (1) $K=\frac{2^7\cdot 3\cdot \pi^2}{6\cdot 21\sqrt{21} \cdot L(2,21)}=\frac{2^7 \cdot 3\pi^2}{6\cdot 21\sqrt{21}}\cdot \frac{441}{8\sqrt{21}\cdot \pi^2}=8$. So, if $D=21$ corresponds to the free algebra of modular forms, then $K=8$. It means that this algebra is generated by three forms, each of them has weight 2.
\end{proof}

\section*{Estimation of volumes, d=3 mod 4}
\begin{lemma} Let $e$ be a primitive root of $L$. 
The maximal volume $\langle e^\perp \rangle$ $$Vol_{max}(\langle e^\perp \rangle)=\begin{cases} 
\frac{1}{24}\frac{p_{i_1}+1}{2} \cdot \ldots \cdot \frac{p_{i_m}+1}{2},&\text{if $det(\langle e^\perp \rangle)=-2d_e$;}\\
\frac{1}{16}\frac{p_{i_1}+1}{2} \cdot \ldots \cdot \frac{p_{i_m}+1}{2},&\text{if $det(\langle e^\perp \rangle)=-4d_e$;}\\
\frac{1}{16}\frac{p_{i_1}+1}{2} \cdot \ldots \cdot \frac{p_{i_m}+1}{2},&\text{if $det(\langle e^\perp \rangle)=-8d_e$}
\end{cases}$$
\end{lemma}
\begin{proof}
The first equality follows from lemma 22.\\
The second equality follows from the fact that all local densities are the same as in the first case, except for $a_2=2^6(1-2^{-2})(1\pm 2^{-1})^{-1}$. The volume is estimated from the above, so we choose the sign plus. Substituting into the volume formula from lemma 22, we get the second statement.\\
The third equality also follows from the fact that all densities are the same as in the first case except for $a_2=2^8(1-2^{-2})(1\pm 2^{-1})^{-1}$. As before, we choose plus to get the upper bound and substitute this in the volume formula. The lemma follows.
\end{proof}

The last equality from this lemma can be made more precise:

\begin{lemma}
Let $e$ be a primitive root of $L$ such that $e^2=-2e',\ det(\langle e^\perp\rangle)=-8d_e, \ e'=3 \pmod{4}$. Then the maximal volume of $\langle e^\perp \rangle$ $Vol_{max}(\langle e^\perp \rangle)=\frac{1}{48}\frac{p_{i_1}+1}{2} \cdot \ldots \cdot \frac{p_{i_m}+1}{2}$
\end{lemma}
\begin{proof}
We take the same estimation for all local densities, except for $a_2$, as in the previous lemma. By lemma 14.c) we know the explicit form of $\langle e^\perp \rangle$ over $\mathbb{Z}_2$. In this case the local density $a_2=2^8(1-2^{-2})(1- 2^{-1})^{-1}$. In the last bracket the sign minus is chosen because the even component $\langle e^\perp \rangle$ is not equivalent to the hyperbolic plane. We substitute the local densities into the volume formula and the lemma follows.
\end{proof}

\begin{lemma}
$Vol(O(L)) = \frac{d^{3/2}}{2^{k+3}\cdot 3\pi^2} L(2,4d)$.
\end{lemma}
\begin{proof}
The local densities are:\\
$a_2=2^7(1-2^{-2})$\\
$a_{p_i}=2p_i (1-p_i^{-2})$, where $p_i$ is the divisor of $d$\\
$a_p=(1-p^{-2})(1-\left( \frac{4d}{2}\right) p^{-2})$\\
We substitute this into the volume formula and the lemma follows.
\end{proof}

\subsection*{The proof of theorem 4}
\begin{proof}

We need to find the lower bound for the left hand side of formula (1) like in the proof of theorem 3.
$K\geq 8$, $L(s,\chi)\geq \zeta(2s)/\zeta(s)$. That's why $\frac{d^{3/2}}{2^{k+3}\cdot 3\pi^2} L(2,4d)\cdot K  \geq \frac{d^{3/2}}{2^{k+3}\cdot 3\pi^2} \frac{\zeta_Q(4)}{\zeta_Q(2)} \cdot 8=\frac{d^{3/2}}{2^{k}\cdot 3^2\cdot 5}$.

We need to find the upper bound for the left hand side of formula (1): $\leq \frac{1}{2}(\frac{1}{24}+\frac{1}{16}+\frac{1}{16})\frac{1}{2^k}(p_1+3)\cdot \ldots \cdot (p_k+3) = \frac{1}{12}\frac{1}{2^k}(p_1+3)\cdot \ldots \cdot (p_k+3)$.

So, we need to compare $(p_1\cdot \ldots \cdot p_k)^{3/2}$ and $\frac{15}{4}(p_1+3)\cdot \ldots \cdot (p_k+3)$. \\

Consider the set of k-tuples of primes such that their product is equivalent to 3 mod 4, arranged in the ascending order. We equip this set with a componentwise order like in the proof of theorem 3. 

Assume $k=4$. The minimal 4-tuple is $3,\ 5,\ 7,\ 11$. For this tuple the l.h.s. is greater than the r.h.s., so each $d$ with at least 4 prime divisors can't correspond to free algebra of modular forms.

Assume $k=3$. For 3-tuples $3,\ 7,\ x$ and 3-tuples, majorizing them, the l.h.s. is greater than the r.h.s. The remaining 3-tuples have the form $3,\ 5,\ x$, where $x<15$ (otherwise the l.h.s. is greater than the r.h.s.). So, there is only one remaining case:
$$ d=3\cdot 5 \cdot 13$$

Assume $k=2$. For 2-tuples which have the form $7,\ x$ the l.h.s. is greater than the r.h.s. The other cases have the form $3,\ x$; $5,\ y$; $x\leq 23,\ y\leq 11$. So, when $k=2$ the remaining cases are
$$ d=3\cdot 5,\ d=3\cdot 13,\ d=3\cdot 17,\ d=5\cdot 7,\ d=5\cdot 11.$$

Assume $k=1$. When $p\geq 19$ the l.h.s. is greater than the r.h.s., so the remaining cases are $3$, $7$, $11$. 

Lemma 16 means that our estimates can be improved:

Assume $d=3\cdot 5 \cdot 13$, then there are no roots $e$ such that $e'=3,\ e''=3\cdot 5,\ e'''=3\cdot 13$. Since we considered their volumes in the r.h.s. of (1), it should be reduced by $$Vol(\langle e'^\perp \rangle)+Vol(\langle e''^\perp \rangle)+Vol(\langle e'''^\perp \rangle)=$$ $$=\frac{1}{24}\left(\frac{3+1}{2} + \frac{3+1}{2} \cdot \frac{5+1}{2} + \frac{13+1}{2} \cdot \frac{3+1}{2}\right)=\frac{11}{12};$$

Assume $d=3\cdot 5$, then there are no roots $e$ such that $e'=3,\ e''=3\cdot 5$. Since we considered their volumes in the r.h.s. of (1), it should be reduced by $$Vol(\langle e'^\perp \rangle)+Vol(\langle e''^\perp \rangle)=\frac{1}{24}\left( \frac{3+1}{2} + \frac{3+1}{2}\cdot \frac{5+1}{2}\right)=\frac{1}{3};$$

Assume $d=3\cdot 13$, then there are no roots $e$ such that $e'=3,\ e''=3\cdot 13$. Since we considered their volumes in the r.h.s. of (1), it should be reduced by  $$Vol(\langle e'^\perp \rangle)+Vol(\langle e''^\perp \rangle)=\frac{1}{24}\left( \frac{3+1}{2} + \frac{3+1}{2}\cdot \frac{13+1}{2}\right)=\frac{2}{3};$$

Assume $d=3\cdot 17$, then there are no roots $e$ such that  $e'=3,\ e''=3\cdot 17$. Since we considered their volumes in the r.h.s. of (1), it should be reduced by $$Vol(\langle e'^\perp \rangle)+Vol(\langle e''^\perp \rangle)=\frac{1}{24}\left( \frac{3+1}{2} + \frac{3+1}{2}\cdot \frac{17+1}{2}\right)=\frac{5}{6};$$

Assume $d=5\cdot 7$, then there are no roots $e$ such that $e'=7,\ e''=5\cdot 7$. Since we considered their volumes in the r.h.s. of (1), it should be reduced by $$Vol(\langle e'^\perp \rangle)+Vol(\langle e''^\perp \rangle)=\frac{1}{24}\left( \frac{7+1}{2} + \frac{7+1}{2}\cdot \frac{5+1}{2}\right)=\frac{2}{3};$$

Assume $d=5\cdot 11$, then there are no roots $e$ such that $e'=11,\ e''=5\cdot 11$. Since we considered their volumes in the r.h.s. of (1), it should be reduced by $$Vol(\langle e'^\perp \rangle)+Vol(\langle e''^\perp \rangle)=\frac{1}{24}\left( \frac{11+1}{2} + \frac{11+1}{2}\cdot \frac{5+1}{2}\right)=1;$$

Assume $d=3$, then there is no root $e$ such that $e'=3$. Since we considered this volume in the r.h.s. of (1), it should be reduced by $$Vol(\langle e'^\perp \rangle)=\frac{1}{24} \frac{3+1}{2}=\frac{1}{12};$$

Assume $d=7$, then there is no root $e$ such that $e'=7$. Since we considered this volume in the r.h.s. of (1), it should be reduced by $$Vol(\langle e'^\perp \rangle)=\frac{1}{24}\frac{7+1}{2}=\frac{1}{6};$$

Assume $d=11$, then there is no root $e$ such that $e'=11$. Since we considered this volume in the r.h.s. of (1), it should be reduced by $$Vol(\langle e'^\perp \rangle)=\frac{1}{24}\frac{11+1}{2}=\frac{1}{4}.$$

We notice that when $det \langle e^\perp \rangle= -8d_e$, $e=-2e'$, $e'=3 \pmod{4}$ we use the estimate for $Vol_{max}(\langle e^\perp \rangle)$ from lemma 20. It follows from lemma 21 that in fact such volumes are at least three times less that our estimate, so we should also subtract from the r.h.s. $\frac{2}{3}Vol_{max}(\langle e^\perp \rangle)=\frac{2}{3}\frac{1}{16}\frac{p_{i_1}+1}{2} \cdot \ldots \cdot \frac{p_{i_m}+1}{2}=\frac{1}{24}\frac{p_{i_1}+1}{2} \cdot \ldots \cdot \frac{p_{i_m}+1}{2}$
for each such root $e$. These roots and corresponding volumes coincide with the ones that we have just computed.

Subtracting this from the r.h.s. we obtain that for all discriminants except $d=3\cdot 5$, $d=3$, $d=7$ the l.h.s. of formula (1) is greater than the r.h.s., that's why the corresponding algebras of modular forms can't be free.

Consider these cases in more details using lemma 17. 

Assume $d=15$, then there are no roots $e=-2e'$ such that $e'=5$ (because $3t^2=-1 \pmod{5}$ is not solvable) and $e'=15$ (because $t^2=-1 \pmod{3}$ is not solvable). Besides, there are no roots $e=-4e'$ such that $e'=3$ (because $5t^2=-2 \pmod{3}$ is not solvable) and $e'=15$ (because $t^2=-2\pmod{5}$ is not solvable). So, the r.h.s. of formula (1) is equal to $Vol(\langle-4^\perp\rangle)+Vol(\langle-4\cdot 5\perp\rangle)+Vol(\langle-2\cdot 3^\perp\rangle,det\langle-2\cdot 3\perp\rangle=-8\cdot 5)+Vol(\langle-2^\perp\rangle,det\langle-2\perp\rangle=-2\cdot 3\cdot 5)+Vol(\langle-2^\perp\rangle,det\langle-2\perp\rangle=-8\cdot 3\cdot 5)$ $ =\frac{1}{16}+\frac{1}{16}\cdot \frac{5+1}{2}+\frac{1}{48}\cdot \frac{3+1}{2}+\frac{1}{24}+\frac{1}{16}=\frac{19}{48}$. Comparing it with the estimate for the l.h.s. we get that the corresponding algebra of modular forms can't be free.

Assume $d=7$, then there are no roots $e=-2e'$ such that $e'=7$ (because $t^2=-1 \pmod{7}$ is not solvable). Besides, there are no roots $e=-4e'$ such that $e'=7$ (because $t^2=-2 \pmod{7}$ is not solvable). So, the r.h.s. of formula (1) is equal to $Vol(\langle-4^\perp\rangle)+Vol(\langle-2^\perp\rangle,det\langle-2\perp\rangle=-2\cdot 7)+Vol(\langle-2^\perp\rangle,det\langle-2\perp\rangle=-8\cdot 7)$ $ =\frac{1}{16}+\frac{1}{24}+\frac{1}{16}=\frac{1}{6}$. Comparing it with the estimate for the l.h.s. we get that the corresponding algebra of modular forms can't be free.

Assume $d=3$, then there are no roots $e=-2e'$ such that $e'=3$ (because $t^2=-1 \pmod{3}$ is not solvable). So the r.h.s. of formula (1) is equal to $Vol(\langle-4\cdot 3^\perp\rangle)+Vol(\langle-4^\perp\rangle)+Vol(\langle-2^\perp\rangle,det\langle-2\perp\rangle=-2\cdot 3)+Vol(\langle-2^\perp\rangle,det\langle-2\perp\rangle=-8\cdot 3)$ $ =\frac{1}{16}\cdot \frac{3+1}{2}+\frac{1}{16}+\frac{1}{24}+\frac{1}{16}=\frac{7}{24}$. Computing the $K$, we get $K=14$.
\end{proof}

\section*{Estimation of volumes, d=2 mod 4}
\begin{lemma} Let $e$ be a primitive root of $L$. 
The maximal volume $\langle e^\perp \rangle$ $$Vol_{max}(\langle e^\perp \rangle)=\begin{cases} 
\frac{1}{24}\frac{p_{i_1}+1}{2} \cdot \ldots \cdot \frac{p_{i_m}+1}{2},&\text{if $det(\langle e^\perp \rangle)=-2d_e$;}\\
\frac{1}{16}\frac{p_{i_1}+1}{2} \cdot \ldots \cdot \frac{p_{i_m}+1}{2},&\text{if $det(\langle e^\perp \rangle)=-4d_e$;}\\
\frac{1}{16}\frac{p_{i_1}+1}{2} \cdot \ldots \cdot \frac{p_{i_m}+1}{2},&\text{if $det(\langle e^\perp \rangle)=-8d_e$}\\
\frac{1}{16}\frac{p_{i_1}+1}{2} \cdot \ldots \cdot \frac{p_{i_m}+1}{2},&\text{if $det(\langle e^\perp \rangle)=-16d_e$}
\end{cases}$$
\end{lemma}
\begin{proof}
The first three equalities follow from lemma 25.\\
To prove the last equality we notice that all the local densities except for $a_2$ are the same as in the first case. Since the lattice is $\langle 2\epsilon_1\rangle \oplus \langle 2\epsilon_2\rangle \oplus \langle -4\epsilon_1 \epsilon_2 d_e \rangle$, $a_2=2^9$. Substituting it into the volume formula, we get the statement of the lemma.
\end{proof}

We notice that if a root $e$ is such that $det(\langle e^\perp \rangle)=-8d_e$, there are two possibilities. Either $\langle e^\perp \rangle \cong U(2) \oplus \langle 2d_e\rangle$, then by lemma 20 there exists the unique glueing vector. Or $\langle e^\perp \rangle \cong \begin{pmatrix}
4 & 2  \\
2 & 4 \\
\end{pmatrix} \oplus \langle -6d_e \rangle$, then by lemma 20 there exist up to 3 glueing vectors, but by lemma 26 the volume of  $\langle e^\perp \rangle$ is at least three times less than in the first case. That's why the upper bound for the contribution of these summands into the r.h.s. of (1) is the same and coincides with the one stated in the previous lemma.

\begin{lemma}
$Vol(O(L)) = \frac{(2d)^{3/2}}{2^{k+3}\cdot 3\pi^2} L(2,8d')$.
\end{lemma}
\begin{proof}
The local densities are:\\
$a_2=2^8(1-2^{-2})$\\
$a_{p_i}=2p_i (1-p_i^{-2})$, where $p_i|d$\\
$a_p=(1-p^{-2})(1-\left( \frac{8d'}{2}\right) p^{-2})$\\
Substituting it into the volume formula, we get the statement.
\end{proof}
\subsection*{The proof of theorem 5}
\begin{proof}

We need to find the lower bound for the left hand side of formula (1) like in the proof of theorem 3. $K\geq 8$, $L(s,\chi)\geq \zeta(2s)/\zeta(s)$. That's why $\frac{(2d')^{3/2}}{2^{k+3}\cdot 3\pi^2} L(2,8d')\cdot K  \geq \frac{(2d')^{3/2}}{2^{k+3}\cdot 3\pi^2} \frac{\zeta_Q(4)}{\zeta_Q(2)} \cdot 8=\frac{(2d')^{3/2}}{2^{k}\cdot 3^2\cdot 5}$.

We need to find the upper bound for the left hand side of formula (1): $\leq \frac{1}{2}(\frac{1}{24}+\frac{1}{16}+\frac{1}{16}+\frac{2}{16})\frac{1}{2^k}(p_1+3)\cdot \ldots \cdot (p_k+3) = \frac{7}{48}\frac{1}{2^k}(p_1+3)\cdot \ldots \cdot (p_k+3)$.

So, we need to compare $(2p_1\cdot \ldots \cdot p_k)^{3/2}$ and $\frac{105}{16}(p_1+3)\cdot \ldots \cdot (p_k+3)$. 

Consider the set of k-tuples of odd primes arranged in the ascending order. We equip this set with a componentwise order like in the proof of theorem 3. The minimal tuple for $k=4$ is $3,\ 5,\ 7,\ 11$. For this tuple the l.h.s. is greater than the r.h.s., so each $d'$ with at least 4 prime divisors can't correspond to free algebra of modular forms.

Assume $k=3$. For $3,\ 5,\ 11$ and 3-tuples, majorizing it, the l.h.s. is greater than the r.h.s. The only remaining 3-tuples is $3,\ 5,\ 7$. $L(2,8\cdot 3 \cdot 5 \cdot 7)=\frac{31}{4410}\sqrt{210}\pi^2$, and substituting into the formula (1) the exact value instead of an estimate, we get that this case is impossible.

Assume $k=2$. For 2-tuples $3,y,\ y\geq 13$ and $5,\ x$  and 2-tuples, majorizing them, the l.h.s. is greater than the r.h.s. The rest tuples have the form $3,\ x$, where $x\leq 11$. $L(2,8\cdot 3 \cdot 5)=\frac{17}{900}\sqrt{30}\pi^2$, $L(2,8\cdot 3 \cdot 7)=\frac{3}{196}\sqrt{42}\pi^2$, $L(2,8\cdot 3 \cdot 11)=\frac{14}{1089}\sqrt{66}\pi^2$, substituting the exact values into the r.h.s. of (1), we get that the cases $3\cdot 7$ and $3\cdot 11$ are impossible. So the remaining 2-tuple is $ d'=3\cdot 5$

Assume $k=1$. When $p\geq 11$ the l.h.s. is greater than the r.h.s. so the remaining primes are $3$, $5$, $7$. $L(2,8\cdot 3)=\frac{1}{24}\sqrt{6}\pi^2$, $L(2,8\cdot 5)=\frac{7}{200}\sqrt{10}\pi^2$, $L(2,8\cdot 7)=\frac{5}{196}\sqrt{14}\pi^2$. Substituting these into the r.h.s. of (1) we get that $d'=7$ is impossible.

Assume $k=0$, it corresponds to the case $D=8$. $L(2,8)=\frac{1}{16}\sqrt{2}\pi^2$. Substituting this into the r.h.s. of (1) we get that the equality is possible if $K=14$.

We consider the rest cases in more details applying lemma 21.

Assume $d'=3\cdot 5$, then there are no roots $e=-2e'$ such that $e'=3,\ e''=3\cdot 5$, because $10t^2=-1\pmod{3},\ 2t^2=-1\pmod{5}$ is not solvable. Besides, there are no roots $e=-4e'$ such that $e'''=5,\ e''''=3\cdot 5$, because $3t^2=-1\pmod{5},\ t^2=-1\pmod{3}$ is not solvable. Since we considered this volume in the r.h.s. of (1), it should be reduced by $$ Vol(\langle e'^\perp \rangle)+Vol(\langle e''^\perp \rangle)+Vol(\langle e'''^\perp \rangle)+Vol(\langle e''''^\perp \rangle)=$$ $$\frac{3}{16}\left( \frac{3+1}{2} + \frac{3+1}{2}\cdot \frac{5+1}{2}\right) + \left(\frac{1}{16}+\frac{1}{24}\right)\left( \frac{5+1}{2} + \frac{3+1}{2}\cdot \frac{5+1}{2}\right)=\frac{39}{32};$$

Assume $d'=5$, then there is no root $e=-2e'$ such that  $e'=5$, because $2t^2=-1 \pmod{5}$ is not solvable. Since we considered this volume in the r.h.s. of (1), it should be reduced by  $$ Vol(\langle e'^\perp \rangle)= \frac{3}{16}\cdot \frac{5+1}{2}=\frac{9}{32};$$

Assume $d'=3$, then there is no root $e=-4e'$ such that $e'=3$, because $t^2=-1 \pmod{3}$ is not solvable. Since we considered this volume in the r.h.s. of (1), it should be reduced by $$ Vol(\langle e'^\perp \rangle)= \left( \frac{1}{24} + \frac{1}{16} \right)\frac{3+1}{2}=\frac{5}{48};$$

Subtracting this from the r.h.s. we obtain that for all discriminants except $d'=3$ and $d'=1$ the l.h.s. of formula (1) is greater than the r.h.s., that's why the corresponding algebras of modular forms can't be free. If the algebra corresponding to $d'=3$ if free, then $K\leq 10$. If the algebra corresponding to $d'=1$ if free, then $K=14$.

\end{proof}

\section*{Conclusion}
We compare our results with the known ones.

Assume $D=5$, then the corresponding algebra of symmetric modular forms of even weight is free \cite{gun}. The weights of the generators are 2, 6, 10, so $K=2+6+10+2=20$, and that agrees with our computations.

Assume $D=8$, then the corresponding algebra of symmetric modular forms of even weight is free \cite{mul}.  The weights of the generators are 2, 4, 6, so $K=2+4+6+2=14$, and that agrees with our computations.

Assume $D=13$, then \cite{geer-zagier} the corresponding algebra of symmetric modular forms of even weight is $\mathbb{C}[A,B,C,C']/\langle R \rangle$, where the weights of the generators are 2, 4, 6, 6. It means that $\dim M_2^{sym}=1<3$, so $K>8$. So this algebra is not free.

Assume $D=24$, then it follows from $K\leq 10$ that $\dim M_2 \leq 2$, because the weights of the generators must be 2, 2, 2 or 2, 2, 4. By theorem 5 from \cite{geer1} the algebra of modular forms of even weights is generated by generators with weights  2, 4, 4, 6, 6, 6, 6, 6, 8, 8, 8, 8 and 10 (mod some relations). So, $\dim M_{2}=1$, and it means that the corresponding algebra of symmetric modular forms of even weights can't be free.

Assume $D=21$. Since $K=8$, for the algebra of symmetric modular forms of even weights to be free it is necessary that $\dim M_{2k}^{sym}= 3$. Since $\dim M_{2k}^{sym}\leq \dim M_{2k}$, it suffices to show that $\dim M_{2k}\leq 2$, to obtain that this algebra is not free.
In\cite{geer} we can find the dimension formula for the space of parabolic forms  $S_{2k}(\Gamma)$ of weight $2k$. 
$$ \dim S_{2k}(\Gamma)=k(k-1)vol(H^2/\Gamma)+\chi (Y_\Gamma)-\begin{cases} 
\frac{1}{3}a_3^{+},&\text{if $k=2 \pmod{3}$;}\\
0,&\text{otherwise}.
\end{cases} $$
In our case $\Gamma=PSL_2(O_K)$. Then $\dim(M_{2k})=\dim S_{2k}+h=\dim S_{2k}+1$, where $h$ is the class number, $\mathbb{Q}(\sqrt{21})$ has class number one. In the table in the end of the book we find that $\dim S_{2}=\chi(Y_\Gamma)=1$ (in the notations of the book the ideal $\mathfrak{a}=O_K$, and $\gamma=(+,+)$ -- the principal genus, corresponding to it). So, $\dim M_{2}=2$, which means that the corresponding algebra of symmetric modular forms of even weight can't be free.

\end{document}